\documentclass[12pt]{amsart}

\usepackage{graphics, amsmath, amsfonts, amssymb, amscd, mathrsfs}

\newtheorem{theorem}{Theorem}
\newtheorem{proposition}[theorem]{Proposition}
\newtheorem{lemma}[theorem]{Lemma}
\newtheorem{corollary}[theorem]{Corollary}

\theoremstyle{definition}
\newtheorem{definition}{Definition}

\newcommand{\B}{\mathscr{B}}
\newcommand{\C}{\mathbb{C}}
\newcommand{\D}{\mathbb{D}}
\renewcommand{\O}{\mathscr{O}}
\newcommand{\R}{\mathbb{R}} 

\newcommand{\T}{\mathbb{T}}

\title{A disc formula for  plurisubharmonic subextentions \\ And a characterization of Thinness of subsets of $\C^n$}

\author{Djire Ibrahim K.}
\address{Jagiellonian University, Department of Mathematics}
\email{Ibrahim.Djire@im.uj.edu.pl}

\keywords{ Analytic discs, plurisubharmonic functions, non-thin set, maximum principle, envelope of functional}

\begin{document}
\begin{abstract}
Let $X\subset \C^n$ be a domain and  $W\subset X$ be a subdomain, $X\neq W$. Suppose that $\varphi_1$ is upper semicontinuous in $X\setminus \overline W$ and $\varphi_2$ is upper semicontinuous in $W.$  We define $\varphi: X\longrightarrow \overline\R$ by  $\varphi= \varphi_1$ in $X\setminus \overline{W}$ , $\varphi= \min\{\varphi_1^* , \varphi_2^*\}$ on $X\cap\partial W$ and $\varphi= \varphi_2$ in $W$. Under suitable conditions on $W$ and $X$,  we will  prove that 
$$EH(x)=  \inf\left \{ \frac{1}{2\pi}\int_0^{2\pi}\varphi\circ f(e^{i\theta})d\theta, f\in\O(\overline{\D},X),  f(0)=x  \right \} $$ is the largest plurisubharmonic function on $X$ less than $\varphi.$\\
In case where $\varphi_1^*=\varphi_2^*$ on $X\cap\partial W$ we get the classical result of Poletsky. 
In case where $\varphi_2$ is big  enough in $W$ our work looks as a subextention result of Larusson-Poletsky. In some sense we have a generalization of these results.
Our disc formula can also be  seen as a generalization of Poletsky's classical result to the situation where the kernel of Poisson functional is not upper semicontinuous on $X.$ At the end we will  characterize the $thinness$ of a set at a point  with  closed  analytic discs and give a version of maximum principle for certain $non$-$thin$ sets in $\C^n.$
\end{abstract}
\maketitle

\section{Introduction}
\noindent
The main goal of the theory of disc functionals is to provide disc formulas for important extremal plurisubharmonic functions in pluri\-potential theory, that is, to describe these functions as envelopes of disc functionals.  This brings the geometry of analytic discs into play in pluripotential theory.  Disc formulas have been proved for largest plurisubharmonic minorants  in ([3], [2], [1]).

Let  $X$ be a domain in  complex affine space $\mathbb{C}^n$ and $W\subset X$ a subdomain, $X\neq W$. Consider two upper semicontinuous functions $\varphi_1: X\setminus \overline W \longrightarrow\R $ and $\varphi_2: W\longrightarrow\R $. We define $\varphi: X\longrightarrow \R$ by $\varphi= \varphi_1$ in $X\setminus \overline{W}$ , $\varphi= \min\{\varphi_1^* , \varphi_2^*\}$ on $X\cap\partial W$ and $\varphi= \varphi_2$ in $W$. Notice that $\varphi$ is not necessarily  upper semicontinuous on $X$.  We take  the constant function $-\infty$ to be plurisubharmonic.\\
\smallskip\noindent
{\it Acknowledgement.}  I wish to thank my  supervisor Professor Armen Edigarian and my colleague Mr. Dongwei Gu for interesting discussions.  
\section{Disc formula}
Let $\D$ denote the unit disc,  $\T$ the unit circle and $\sigma$ the arc length measure on $\T$.
Set
 $$\B_2=\{f\in\O(\overline{\D},X), f(\T)\subset X\setminus \overline{W}\}$$
 $$\B_1=\{f\in\O(\overline{\D},X), f(\T)\subset W\}.$$
Set $\B=\B_1\cup\B_2$. Assume that for all $x\in X$ there is $f\in \B$ such that $f(0)=x$. We define $$
F(x)=\inf\left\{\frac{1}{2\pi}\int_0^{2\pi}\varphi\circ f(e^{i\theta})d\theta, f\in\B, f(0)=x\right\}.$$
\begin{proposition}
If for all $x\in X$ there is   $f\in\B$ so that $f(0)=x,$ then
$F$ is upper semicontinuous on $X$.
\end{proposition}
\begin{proof}
Let $c\in \R$, $x\in X$ such that $F(x)<c.$ We will prove that there is a neighborhood $V$ of $x$ such that $F(y)<c$ for all $y\in V.$  By definition of $F$ there is $f_0\in\B$ with $f_0(0)=x$ such that $ \frac{1}{2\pi}\int_0^{2\pi}\varphi\circ f_0(e^{i\theta})d\theta<c.$ Assume that $f_0\in\B_1$ then $f_0(\T)\subset W.$ As $\varphi$ is upper semicontinuous on $W$ we can find a decreasing sequence of continuous functions $(\psi_j)_j$ defined on $W$ which converges to $\varphi$. There is  $j_0>1$ such that $ \frac{1}{2\pi}\int_0^{2\pi}\psi_{j_0}\circ f_0(e^{i\theta})d\theta <c. $ As $W$ is open and $\psi_{j_0}$ is continuous then one can find $V \subset\subset X$ a small neighborhood of $x$ such that $\{f_0(\T)+y-x, y\in V \} \subset\subset W$,  $\{f_0(\overline{\D})+y-x, y\in V \} \subset\subset X$ and $\frac{1}{2\pi}\int_0^{2\pi}\psi_{j_0}( f_0(e^{i\theta})+y-x)d\theta<c$ for all $y$. Hence $F(y)<c$ for all $y\in V$. 
\end{proof}
\noindent
Notice that on $X\cap \partial W$ we may have $\varphi<F$ since $F(x)=\lim_{r\rightarrow 0}\sup_{B(x,r)} F.$ For example take $\varphi_1=2$ and $\varphi_2=-1$. If $W\subset\subset X$ then we get $F=2$ on $\partial W$ while $\varphi=-1$ there.  Now consider the functions 
$$PF(x)=\inf\left \{ \frac{1}{2\pi}\int_0^{2\pi}F\circ f(e^{i\theta})d\theta, f\in\O(\overline{\D},X),  f(0)=x  \right \}. $$
 $$EH(x)=  \inf\left \{ \frac{1}{2\pi}\int_0^{2\pi}\varphi\circ f(e^{i\theta})d\theta, f\in\O(\overline{\D},X),  f(0)=x  \right \} .$$
 Recall that $F$ is upper semicontinuous, then by Poletsky's classical theorem we have $$PF=\sup\{v\in PSH(X), v\leq F\}$$ is the largest plurisubharmonic and less than $F.$ Our goal here is to prove that in $X$  $$EH=PF\leq \varphi.$$ 
 Remark that by definitions $EH\leq F$ and $EH\leq \varphi.$ If we reach to prove $PF\leq\varphi$ then we get $PF\leq EH.$
\noindent
 The following result due to (Bu-Schachermayer) is the core of the proof of Lemma 3. For a detailed proof of the lemma below see $[1]$. 
\begin{lemma}[Bu-Schachermayer]
Let $A$ be a compact subset of $\T$ and $\psi\in C(\overline{\D})$. Then there exists a sequence $(p_k)$ of polynomials  $p_k:\C \rightarrow \C$ satisfying
\begin{itemize}
\item[(i)] $p_k(\D)\subset \D$ and $p_k(0)=0,$
\item[(ii)] $p_k \rightarrow 0$ uniformly on every compact subset of $\overline{\D}\setminus A$,
\item[(iii)] $\int_{A} \psi \circ p_k(t)d\sigma(t) \longrightarrow \sigma(A)\int_{\T}\psi(t)d\sigma(t)$.
\end{itemize}
\end{lemma}
\noindent
\begin{lemma}If for all $x\in X$ there is   $f\in\B$ so that $f(0)=x$,
 then $EH \leq  PF$. 
 \end{lemma}
\begin{proof} 
Let  $x\in X$, $h\in\O(\overline{\D},X)$, with $h(0)=x$, $\epsilon>0$ and $\psi$ a continuous function on $X$ bigger than $F$. We will prove that there is $G\in\O(\overline{\D},X)$, $G(0)=x$ such that $$  
\int_{\T}\varphi\circ G (t) d\sigma(t)\leq \int_{\T} \psi\circ h(t) d\sigma(t). $$
 Let $t_0\in \T$ there is $g_0\in\O(\overline{\D},\C^n)$ such that $g_0(0)=0$, $h(t_0)+g_0\in\B$ and
\begin{equation}
  \int_0^{2\pi}\varphi ( h(t_0)+g_0(z))d\sigma(z)< \psi\circ h(t_0)+\epsilon/2.
 \end{equation}
 We may assume that the map $h(t_0)+g_0$ belongs to $\B_1$. Take a continuous function $B_0$ on $W$ bigger than $\varphi$ such that
 \begin{equation}
  \int_0^{2\pi}B_0 ( h(t_0)+g_0(z))d\sigma(z)<\psi\circ h(t_0)+\epsilon.
  \end{equation}
 Extend $B_0$ to a continuous function on $X.$  As $W$ is open, $B_0$ and $\psi$ are continuous, then there is an open arc $I_0$ containing $t_0$ such that 
 $$  \{ h(t)+g_0(z), t\in I_0, z\in\overline{\D} \} \subset\subset X \mbox{  and } \{ h(t)+g_0(z), t\in I_0, z\in \T \} \subset\subset W. $$
 $$|B_0(h(t)+g_0(z))-B_0(h(t_0)+g_0(z))|<\epsilon\mbox{  for } t\in I_0, z\in\T,$$
 and
 $$  |\psi\circ h(t)-\psi\circ h(t_0)|<\epsilon, \mbox{ for } t\in I_0.$$
 By compactness there is $N>0$, points $t_1,..., t_N \in \T$, open  arcs $I_1,....,I_N$, holomorphic maps $g_1,..., g_N \in\O(\overline{\D},\C^n)$ and $B_1,...., B_N$ continuous functions $X$ bigger than $\varphi$  either on $W$ or on $X\setminus\overline{W}$ such that
$$t_j\in I_j, g_j(0)=0, h(t_j)+g_j\in\B \mbox{ and } \T\subset\cup I_j,$$
and
$$  \{ h(t)+g_j(z), t\in I_j, z\in\overline{\D} \} \subset\subset X\mbox{and  }  \{ h(t)+g_j(z), t\in I_j, z\in\T \} \subset\subset (W\mbox{ or } X\setminus \overline{W}) $$
\begin{equation}
  \int_0^{2\pi}B_j ( h(t_j)+g_j(z))d\sigma(z)<\psi\circ h(t_j)+\epsilon,
\end{equation}
\begin{equation}
 |B_j(h(t)+g_0(z))-B_j(h(t_j)+g_0(z))|<\epsilon\mbox{  for } t\in I_j, z\in\T,
\end{equation}
\begin{equation}
 |\psi\circ h(t)-\psi\circ h(t_j)|<\epsilon, \mbox{ for } t\in I_j.
 \end{equation}
Choose $\delta_0$ very small such that  for all $j$

$$
\left\{h(t)+g_j(z)+ x , ||x||<\delta_0, t\in I_j, z\in \overline\D \right\} \subset\subset X,
$$
$$
\left\{h(t)+g_j(z)+ x , ||x||<\delta_0, t\in I_j, z\in \T \right\} \subset\subset ( W \mbox{  or } X \setminus\overline{W}) 
$$
 and $K\subset\subset X$ an open set containing 
$$ \bigcup_{j=1}^{N}\left\{h(t)+g_j(z)+ x , ||x||<\delta_0, t\in I_j, z\in \overline{\D}\right\}\cup h(\overline{\D}).$$
Take $C> \sum_j\sup_{\overline{K}} |B_j|+\sup_{\overline{K}}|\psi|+|\sup_{\overline{K}}\varphi|$ and disjoint closed arcs $J_j\subset I_j$ such that 
\begin{equation}
C\sigma (\T\setminus \cup_j J_j)<\epsilon.
\end{equation}
By uniform continuity of $B_j$ on $\overline{K}$ there is $0<\delta<\delta_0$ such that
\begin{equation}
|B_j(x_1)-B_j(x_2)|<\epsilon, \mbox{ for all } x_1, x_2  \in K \mbox {  with }  ||x_1-x_2||<\delta,  \mbox {  for all  } j. \end{equation}
Take a small open neighborhood  $V_j$ of $J_j$ such that $$(\cup^N_{{i=1}, {i\neq j}} J_i)\cup\{0\}\subset \C\setminus V_j.$$
Set $K_j=\overline{\D}\setminus V_j$.
By Lemma 2  for each $i=1,...,N$ there is a polynomial $p_i$ such that: \\
\begin{itemize} 
  \item $p_i(0)=0$ and $p_i(\D)\subset \D$, 
  \item $||g_i\circ p_i(z)||< \delta/N$ for all $z\in K_i,$ and
 \item $ \int_{J_i} B_i(h(t_i)+g_i\circ p_i(t))d\sigma(t)<  \sigma(J_i)\int_{\T}B_i(h(t_i)+ g_i(t))d\sigma(t)+\epsilon/N$.
  \end{itemize}
Set $$G(z)= h(z)+\sum_{i=1}^N g_i\circ p_i(z) \mbox{ for all  } z\in \overline{\D}.$$
Then $G\in \O(\overline{\D},K)$ and $G(0)=h(0)$ and we have 
\begin{align*}
\int_{\T}\varphi\circ G (t) d\sigma(t) &\leq  \sum_{i=1}^N \int_{J_i} \varphi\circ G(t)d\sigma(t)+\epsilon  \mbox{  \;\;\;\; ''because of  (6)} \\
                                                            &= \sum_{i=1}^{N} \int_{J_i} \varphi \left( h(t)+g_i\circ p_i(t)+\sum_{j=1, i\neq j}^N g_j\circ p_j(t) \right)d\sigma(t) +\epsilon \\  
                                                              &\leq  \sum_{i=1}^{N} \int_{J_i} B_i \left( h(t)+g_i\circ p_i(t)+\sum_{j=1, i\neq j}^N g_j\circ p_j(t) \right)d\sigma(t) +\epsilon \\
                               &\leq \sum_{i=1}^N \int_{J_i} B_i\left( h(t)+g_i\circ p_i(t)\right)d\sigma(t)+2\epsilon  \mbox{  \;\;\;\; ''because of (7)}\\
&\leq \sum_{i=1}^N \int_{J_i} B_i\left( h(t_i)+g_i\circ p_i(t)\right)d\sigma(t)+3\epsilon  \mbox{  \;\;\;\; ''   because of (4)}\\
&\leq \sum_{i=1}^N \sigma(J_i)\int_{\T} B_i\left(h(t_i)+g_i(t) \right)d\sigma(t)+4\epsilon \mbox{\;\;\; \;  ''because of Lemma 2}\\
&\leq  \sum_{i=1}^N \sigma(J_i) \psi\circ h(t_i) +5\epsilon  \;\;\;\; \mbox{   ''because of (3),,}  \\
&\leq  \sum_{i=1}^N \int_{J_i} \psi\circ h(t) d\sigma(t)+6\epsilon   \;\;\;\; \mbox{ \; ''because of  (5)}  \\
&\leq   \int_{\T} \psi\circ h(t) d\sigma(t)+7\epsilon  \mbox{ \;\;\;\;because of (6)}.  
\end{align*}
We get 
$EH(x)\leq \int_{\T} \psi\circ h(t) d\sigma(t)+7\epsilon $ for all $\epsilon>0$, continuous function $\psi\geq F$ and $h\in \O(\overline{\D},X)$ with $h(0)=x.$ Hence $EH(x)\leq PF(x).$
\end{proof}
\noindent
Consider $v:X\rightarrow \R$. Recall that 
$$\limsup_{y\rightarrow x} v(x)=\inf_{r>0}(\sup\{ v(y), y\in \overline B(x,r) \}) \;\;\;\;\; (x\in X).$$
$$\limsup_{{y\rightarrow x},  \;{y\in Y}}v(x)=\inf_{r>0}(\sup\{ v(y), y\in \overline B(x,r)\cap Y\}) \;\;\;\;\; (x\in\overline Y).$$
In what follows   the word $thin$  means $Pluri$-$thin$ or $\C^n$-$thin.$
\begin{definition}
Let $Y$ be a subset of $\C^n$ and $x\in \C^n.$ Then $Y$ is $non$-$thin$ at $x$ if $x\in \overline{ Y\setminus\{x\}}$ and if, for every plurisubharmonic function $u$ defined  on a neighborhood of $x$ one has 
$$ \limsup_{z\rightarrow x,\;\; z\in Y\setminus \{x\}} u(z)= u(x).$$
\end{definition}
\noindent
As example we have, a connected set containing more than one point is $non$-$thin$ at every point of its closure see Theorem 3.8.3 in $[4]$. If $h\in\O(\overline\D,\C^n)$ then the set $h([0,1])$ is not $thin$ at any of its points see Corollary 4.8.5 in $[5].$

\begin{theorem} Let $X\subset \C^n$ be a domain and $W\subset X$. Suppose that
\begin{itemize}
\item[i)]  $\B$ covers $X.$ 
\item[ii)] $X\setminus\overline W$ and $W$ are subdomains of $X$.
\end{itemize}
Then $EH\in PSH(X)$  and 
$$\sup\{v(x), v\in PSH(X), v\leq\varphi\}= \inf\left \{ \frac{1}{2\pi}\int_0^{2\pi}\varphi\circ f(e^{i\theta})d\theta, f\in\O(\overline{\D},X),  f(0)=x  \right \} .$$
\end{theorem}
\begin{proof}
We have $$PF\leq EH\leq PF\leq \varphi.$$
The last inequality, we have $PF\leq F\leq \varphi$ in $X\setminus \partial W$ (because of constant maps in $\B$). Let $x\in\partial W$, we may assume that  $\varphi(x)=\varphi^*_2(x)$. As $PF\in PSH(X)$ and $W$ is $non$-$thin$ at $x$ then $$PF(x)=  \limsup_{z\rightarrow x, \; z\in W} PF(z) \leq     \limsup_{z\rightarrow x, \; z\in W}\varphi_2(z)=\varphi^*_2(x)=\varphi(x). $$ This for all $x\in\partial W.$ Thus $PF\leq \varphi$ on $X.$
The second holds  because of Lemma 3. The first one holds because $PF\in PSH(X)$ and $PF\leq \varphi.$ Hence $PF=EH\in PSH(X).$ Obviously for all $u\in PSH(X)$, $u\leq\varphi$ we have $u\leq EH$ hence $\sup\{v,  v\in PSH(X), v\leq\varphi\}\leq EH.$ As $EH\in PSH(X)$ and less than $\varphi$ then we have equality. \end{proof}
\noindent
As $\varphi$  may  be,  not upper semicontinuous then our formula generalizes Poletsky's classical formula .  For properties of $thin$ sets one can take a look in ($[4]$ and $[5]$).
\section{Thinness}
The following theorem gives a characterization of thinness of an open set at a given point in $\C^n$ in term of analytic discs.
\begin{theorem}
Let $U\subset \C^n$ be open and $x\in \C^n$. Then the following conditions are equivalent 
\begin{itemize}
\item [i)]$U$ is non-thin at $x,$
\item [ii)] For all $\epsilon>0$, all   neighborhood $V$ of $x$ there is $f\in \O(\overline\D, V)$ such that $f(0)=x$ and $$\sigma(\T\cap f^{-1}(V\cap U))>1-\epsilon.$$
\end{itemize}
\end{theorem}
\begin{proof}
$i)\Longrightarrow ii)$\\
Let $\epsilon>0$ and $V$ a  neighborhood of $x$. Let $V_1\subset\subset V$ be an open  and connected neighborhood of $x.$ Then by Poletsky's classical theorem the function $u_{U\cap V_1, V_1}$ is plurisubharmonic in $V_1$, where $$u_{U\cap V_1,V_1}(x)=\inf\{-\sigma(\T\cap f^{-1}(U\cap V_1)), f\in \O(\overline\D,V_1), f(0)=x\}.$$ Since $U$ is $non$-$thin$ at $x$ then $$u_{U\cap V_1, V_1}(x)=\limsup_{{z\rightarrow x, \; z\in U\setminus\{x\}}}u_{U\cap V_1,V_1}(z)=-1,$$ thus there is $f\in \O(\overline{\D},V)$ such that $f(0)=x$ and $$-\sigma(\T\cap f^{-1}(V\cap U))<-1+\epsilon.$$
$i)\Longleftarrow ii)$\\
Let $r_0>0$ and $u\in PSH(B(x,r_0)).$ For any $0<r<r_0$ we set $c_r=\sup\{u(z), z\in \overline B(x,r)\cap U \}$. Take $M>|\sup_{\overline B(x,r)} u|.$ By the  hypothesis  in $ii)$ for any $\epsilon>0$ there is $f_{\epsilon}\in \O(\overline\D, B(x,r))$ with $f_{\epsilon}(0)=x$ such that $$ \sigma(\T\cap f_{\epsilon}^{-1}(B(x,r)\cap U))>1-\epsilon.$$  
 Set $A=\T\setminus(\T\cap f_{\epsilon}^{-1}(U\cap B(x,r)))$ thus we have
 \begin{align*}
 u(x) \leq \int_{\T}u\circ f_{\epsilon}(t)d\sigma(t)&\leq \int_{\T\setminus A}u\circ f_{\epsilon}(t)d\sigma(t)+   \int_{ A}u\circ f_{\epsilon}(t)d\sigma(t)  \\ 
 & \leq c_r(1-\sigma(A))  + M \sigma(A) \leq c_r+ (|c_r|+M)\epsilon.
  \end{align*}
 This for all $\epsilon>0$, hence when $\epsilon \rightarrow 0$  we get $u(x)\leq c_r.$ As $r$ was taken arbitrarily then we have $$u(x)\leq \inf_{r>0} c_r=\limsup_{{z\rightarrow x, \; z\in U\setminus\{x\}}} u(z)\leq        \limsup_{{z\rightarrow x}} u(z)=u(x).$$
 This for all $u$ plurisubharmonic in a neighborhood of $x.$ Hence $U$ is $non$-$thin$ at $x.$
\end{proof}
\noindent
In the light of Corollary 4.8.3 in $ [5]$ we have the following.
\begin{corollary}
Let $Y\subset\C^n$ and $x\in \C^n.$ Then the following conditions are equivalent
\begin{itemize}
\item[i)] $Y$ is $non$-$thin$ at $x,$
\item[ii)] For every $\epsilon>0,$ neighborhood  $V$ of $x$ and every open set $U$ containing $Y\setminus \{x\}$ there exists $f\in \O(\overline\D, V)$ such that $f(0)=x$ and $$\sigma(\T\cap f^{-1}(V\cap U))>1-\epsilon.$$
\end{itemize}
\end{corollary}
\noindent
As a consequence of the corollary above we can remark that if $X$ is a Runge domain and $K\subset X$ is compact then the set of all points at which $K$ is $non$-$thin$ is a subset of the polynomial hull of $K.$\\

For an open set $X\subset \C^n$, $u\in PSH(X)$ and $K\subset X$ compact it is well known (by the classical maximum principle) that $\sup_{K} u=\sup_{\partial K} u.$ We will state a similar result for certain subsets of $X$ not necessarily compact.
\begin{theorem}
Let $U\subset\subset X\subset \C^n$, where $X$ is open. If $U$is $non$-$thin$ at every point of its closure, then for all $u\in PSH(X)$ one has $$\sup_U u=\sup_{\partial U} u.$$
\end{theorem}
\begin{proof} Let $U$ be $non$-$thin$ at every point of its closure. By the classical maximum principle we have  
$$\sup_U u\leq \sup_{\overline U} u=\sup_{\partial \overline U} u\leq \sup_{\partial U} u.$$
Let $x\in\partial U$ and $r>0$ small so that $\overline B(x,r)\subset X.$ As $U$ is $non$-$thin$ at $x$ then $$u(x)= \inf_{\rho>0}\sup_{U\cap\overline B(x,\rho)} u  \leq\sup_{U\cap\overline B(x,r)} u \leq \sup_U u.$$
This for all $x\in \partial U$. Hence $$\sup_{\partial U} u\leq \sup_U u.$$
\end{proof}
\noindent
Actually we have $\sup_U u=\sup_{\partial \overline U} u= \sup_{\partial U} u.$
\begin{corollary}
Let $X\subset\C^n$ be open and $U\subset\subset X$ be a connected subset, then $$ \sup_U u=\sup_{\partial U} u.$$ For all $u\in PSH(X).$
\end{corollary}
\begin{corollary}
Let $X\subset\C^n$ be open and let $U\subset\subset X$ be open. If $U$  is $non$-$thin$ at every point of its closure, then $$u_{U,X}=u_{\overline U,X}.$$
Moreover if $X$ is hyperconvex, then $u_{U,X}$ is continuous in $X$ and $\lim_{z\rightarrow \partial X} u_{U,X}(z)=0.$
\end{corollary}
\begin{proof}
Recall that $$u_{U,X}=\sup\{v\in PSH(X), v<0, v|U\leq -1\},$$ $$u_{\overline U,X}=\sup\{v\in PSH(X), v<0, v|\overline U\leq -1\}.$$ Notice that for all $v\in PSH(X)^-$ with $v\leq-1$ on $\overline U.$ We have $v\leq u_{U,X}$. Thus one has $$u_{\overline U,X}\leq u_{U,X}.$$ As $u_{U,X}\in PSH(X)^-$ and $u_{U,X}=-1$ in $U$, then by Theorem 7 $u_{U,X}=-1$ on $\overline U$. Hence $u_{U,X}$ is in the family defining $u_{\overline U,X}$. Thus $$ u_{U,X}\leq u_{\overline U,X}.$$
If $X$ is hyperconvex then by Proposition 4.5.3 in $[5]$, $u_{\overline U,X}$ is continuous on $X,$ hence $u_{U,X}$ is continuous and by Proposition 4.5.2 we have  $\lim_{z\rightarrow \partial X} u_{U,X}(z)=0.$
\end{proof}
As a consequence we can see that the capacity of a bounded open and connected set coincides with the capacity of its closure.


\begin{thebibliography}{99}
\bibitem{Magnusson2011}
Magn\'usson, B.\ S.  {\it Extremal $\omega$-plurisubharmonic functions as envelopes of disc functionals.}  Arkiv Math.\ \textbf{49} (2011) 383--399.
\bibitem{Poletsky1991}
Poletsky, E.\ A.  {\it  Plurisubharmonic functions as solutions of variational problems.}  Several complex variables and complex geometry (Santa Cruz CA, 1989), 163--171.  Proc.\ Sympos.\ Pure Math., 52, Part 1.  Amer.\ Math.\ Soc., 1991.
\bibitem{ Larusson-Poletsky}
L\'arusson, F.  and Poletsky, E.\ A. {\it Plurisubharmonic subextentions as envelopes of disc functionals.}
\bibitem  l Thomas Ransford; Potential Theory in the Complex Plane.
\bibitem  l Maciej Klimek; Pluripotential Theory.
\end{thebibliography}
\end{document}